\documentclass[11pt, reqno]{amsart}

\usepackage{amssymb,latexsym,amsmath,amsfonts,amsthm}
\usepackage{mathrsfs}
\usepackage{enumitem}
\usepackage[usenames]{color}
\usepackage{hyperref}
\usepackage{comment}

\allowdisplaybreaks

\numberwithin{equation}{section}

\theoremstyle{definition}
\newtheorem{definition}{Definition}[section]

\theoremstyle{remark}
\newtheorem{remark}[definition]{Remark}

\theoremstyle{plain}
\newtheorem{theorem}[definition]{Theorem}
\newtheorem{result}[definition]{Result}
\newtheorem{lemma}[definition]{Lemma}
\newtheorem{proposition}[definition]{Proposition}

\newtheorem{corollary}[definition]{Corollary}

\newcommand{\eps}{\varepsilon}

\newcommand{\OM}{\Omega}

\newcommand{\smoo}{\mathcal{C}}

\newcommand{\sh}{\mathscr{S}}
\newcommand{\iti}{\pi_{{\rm itn}}}
\newcommand{\wrd}{\pi_{{\rm symb}}}
\newcommand{\sta}{{\sf start}}
\newcommand{\varf}{{\sf f}}
\newcommand{\Sig}{\overline{\sigma}}

\newcommand{\bcdot}{\boldsymbol{\cdot}}

\newcommand{\lrarw}{\longrightarrow}

\newcommand{\gen}{\mathscr{G}}
\newcommand\Add[1]{\sum_{{#1}}\nolimits^{\prime}}
\newcommand{\IVar}{\Gamma^\bullet}
\newcommand{\Fam}{\mathscr{F}}
\newcommand\ordi[1]{{#1}^{\raisebox{-2pt}{$\scriptstyle{{\rm th}}$}}}
\newcommand\pthsp[1]{{}^{\raisebox{-1pt}{$\scriptstyle{{#1}\!}$}}X}
\newcommand\pathsp[1]{{}^{\raisebox{-1pt}{$\scriptstyle{{#1}}$}}\mathbb{P}^1}
\newcommand\pth[1]{{}^{\raisebox{-1pt}{$\scriptstyle{{#1}\!}$}}\mathscr{O}}
\newcommand\pat{\Gamma^{\infty}_{\mathscr{G}}}
\newcommand{\met}{D}
\newcommand{\kd}{\boldsymbol{\delta}}
\newcommand{\wod}{\overline{\boldsymbol{\alpha}}}
\newcommand{\jac}{{\rm Jac}}
\newcommand\gti[1]{\boldsymbol{[}{#1}\boldsymbol{]}}
\newcommand{\A}{A_{\raisebox{2pt}{$\scriptstyle{\beta,\,\wod}$}}}
\newcommand{\E}{\mathfrak{S}_{\mathscr{G}}}
\newcommand{\excp}{\mathcal{E}(\mathscr{G})}

\newcommand{\Z}{\mathbb{Z}}
\newcommand{\N}{\mathbb{N}}

\newcommand{\C}{\mathbb{C}} 
\newcommand{\R}{\mathbb{R}}
\newcommand{\pro}{\mathbb{P}}

\newcommand{\wht}{\widehat}
\newcommand{\bsy}{\mathsf}

\begin{document}

\title[Entropy of correspondences \& rational semigroups]{The entropy of holomorphic
correspondences: exact \\ computations and rational semigroups}

\author{Gautam Bharali}
\address{Department of Mathematics, Indian Institute of Science, Bangalore 560012, India}
\email{bharali@iisc.ac.in}

\author{Shrihari Sridharan}
\address{Indian Institute of Science Education \& Research, Thiruvananthapuram 695551, India}
\email{shrihari@iisertvm.ac.in}

\begin{abstract}
We study two notions of topological entropy of correspondences introduced by Friedland and
Dinh--Sibony. Upper bounds are known for both. We identify a class of holomorphic correspondences
whose entropy in the sense of Dinh--Sibony equals the known upper bound. This provides an exact
computation of the entropy for rational semigroups. We also explore a connection between these
two notions of entropy.  
\end{abstract}


\maketitle

\vspace{-0.6cm}
\section{Introduction, definitions and some results}\label{S:intro}
This paper studies certain semigroups of holomorphic maps. It is motivated,
however, by two related notions of topological entropy\,---\,one of which applies to
meromorphic correspondences on a compact K{\"a}hler manifold, while the other applies, more
generally, to closed relations on a compact metric space. Both notions are thus applicable to
holomorphic correspondences (which we shall define presently) on a compact K{\"a}hler manifold.
The first notion is due to Dinh and Sibony \cite{dinhSibony:ubtemc08} while the second,
introduced much earlier, is due to Friedland \cite{friedland:eam93}. In both
cases, upper bounds for each type of topological entropy were given: by Friedland
in \cite{friedland:eam93} and by Dinh--Sibony in \cite{dinhSibony:ubtemc08}.
However, for either type of entropy, this
upper bound is in general \textbf{strictly greater} than the
actual entropy. In this work, among other things, we identify a natural class of
holomorphic correspondences for which this upper bound equals the entropy of Dinh--Sibony.
\smallskip

We use the word ``natural'' because the above-mentioned correspondences turn out to be correspondences
representing certain semigroups of holomorphic maps. Hence, the main results in this paper will
be stated for these semigroups. To get to these results, we need some definitions.

\begin{definition}\label{D:holCorr}
Let $X_1$ and $X_2$ be two compact, connected complex manifolds of dimension $n$. A
\emph{holomorphic correspondence} from $X_1$ to $X_2$ is a formal linear combination
of the form
\begin{equation}\label{E:stdForm}
 \Gamma = \sum\nolimits_{1\leq j\leq N}\!m_j\Gamma_j,
\end{equation}
where the $m_j$'s are positive integers and $\Gamma_1, \Gamma_2,\dots,\Gamma_N$ are distinct
irreducible complex-analytic subvarieties of $X_1\times X_2$ of pure dimension $n$ that satisfy the following
conditions:
\begin{enumerate}[leftmargin=27pt]
  \item for each $\Gamma_j$ in \eqref{E:stdForm},
  $\left.\pi_1\right|_{\Gamma_j}$ and $\left.\pi_2\right|_{\Gamma_j}$ are surjective;
  \vspace{0.8mm}
  
  \item for each $x\in X_1$ and $y\in X_2$, $\left(\pi_1^{-1}\{x\}\cap \Gamma_j\right)$ and
  $\left(\pi_2^{-1}\{y\}\cap \Gamma_j\right)$ are finite sets for each $j$;
\end{enumerate}
where $\pi_i$ is the projection onto $X_i$, $i=1,2$. 
\end{definition}

Given a holomorphic correspondence $\Gamma$ from $X_1$ to $X_2$, the
set (in terms of the notation in \eqref{E:stdForm}) $\cup_{1\leq j\leq N}\Gamma_j$ is called
the \emph{support} of $\Gamma$, which we denote by $|\Gamma|$.
The data $(m_1,\dots, m_N)$ in \eqref{E:stdForm} are an essential
part of the definition above. We shall elaborate on this below, but a brief
reason is as follows. If $X_1 = X_2 = X$ in Definition~\ref{D:holCorr}, then we say that $\Gamma$ is
a \emph{holomorphic correspondence \textbf{on} $X$}. Two holomorphic correspondences on $X$ can be
composed with each other. It is possible for a correspondence $\Gamma$, even if
$m_1 = \dots = m_N = 1$, to be such that some of the irreducible components of
$\Gamma\circ \Gamma$ occur with multiplicity higher than $1$. The ability to compose two
correspondences introduces the perspective of dynamics to the study of correspondences.
\smallskip

We now introduce the two notions of entropy that we alluded to. We begin with the more general
notion.

\begin{definition}[Friedland, \cite{friedland:eam93,friedland:egsg96}]\label{D:entropy_F}
Let $X$ be a compact metric space and let $\Gamma$ be a closed relation on $X$ (i.e.,
$\Gamma$ is a closed subset of $X\times X$ and the projection $\left.\pi_1\right|_{\Gamma}$
is surjective). Let
\[
  X^{\N}\,:=\,\{(x_0, x_1, x_2,\dots) : x_n\in X, \; n\in \N\}
\]
endowed with the product topology, and let
\[
  \Gamma^\infty\,:=\,\{(x_0, x_1, x_2,\dots) \in X^{\N} : (x_n, x_{n+1})\in \Gamma
  \; \ \forall n\in \N\}.
\]
If $\Gamma^\infty$ is endowed
with the topology that it inherits from $X^{\N}$, then, by definition, the left-shift on $X^{\N}$
induces a continuous map $\sigma: \Gamma^\infty\lrarw \Gamma^\infty$, where
$\sigma: (x_0, x_1, x_2,\dots) \longmapsto (x_1, x_2, x_3,\dots)$. Then
\emph{Friedland's entropy} for $\Gamma$, denoted by $h_F(\Gamma)$, is defined as the
topological entropy, in the sense of Bowen, of $\sigma: \Gamma^\infty\lrarw \Gamma^\infty$
(usually denoted by $h(\sigma)$ in the literature).
\end{definition}
For the sake of completeness we will define the topological entropy
in the sense of Bowen\,---\,see Section~\ref{S:notn_top_ent}, where we examine the latter notion more closely.
For the next definition, we need an alternative presentation of the correspondence introduced
in Definition~\ref{D:holCorr}. We rewrite $\Gamma$ as
\begin{equation}\label{E:stdForm_mult}
  \Gamma\,=\,\Add{1\leq j\leq M}\!\IVar_j,
\end{equation}
where the primed sum indicates that the irreducible subvarieties 
$\IVar_{j}$, $j=1,\dots,M$, are \emph{not necessarily distinct} and are repeated according
to multiplicity that is given by the coefficients $m_1,\dots, m_N$ in \eqref{E:stdForm}.
Therefore, $M = m_1+\dots + m_N$. With this explanation, we give

\begin{definition}[Dinh--Sibony, \cite{dinhSibony:ubtemc08}]\label{D:entropy_D-S}
Let $X$ be compact, connected complex manifold and let $\Gamma$ be a holomorphic
correspondence on $X$. For each $\nu\in \Z_+$, a \emph{$\nu$-orbit of $\Gamma$} is any
tuple of the form
\[
  (x_0, x_1,\dots, x_{\nu}; \alpha_1,\dots, \alpha_{\nu})\in X^{\nu+1}\times \{1,\dots,M\}^{\nu},
\]
where $(x_{j-1}, x_j)\in \IVar_{\alpha_j}$, $j=1,\dots, \nu$, assuming the presentation
\eqref{E:stdForm_mult} for $\Gamma$. Fix a metric $d$ compatible with the topology of $X$.
If $\Fam$ is a family of $\nu$-orbits, we say that $\Fam$ is an
\emph{$(\eps, \nu)$-separated family}, $\eps>0$, if for all pairs of distinct elements
\[
  (x_0, x_1,\dots, x_{\nu}; \alpha_1,\dots, \alpha_{\nu}) \quad\text{and}
  \quad (y_0, y_1,\dots, y_{\nu}; \beta_1,\dots, \beta_{\nu})
\]
of $\Fam$, we have
\begin{equation}\label{E:eps-nu_sep}
  d(x_j, y_j)\,>\,\eps \text{ for some $j=0,1,\dots, \nu$} \quad\text{or}
  \quad \alpha_j\neq \beta_j \quad \text{ for some $j = 1,\dots, \nu$}.
\end{equation}
Then, the \emph{topological entropy of $\Gamma$}, denoted by $h_{top}(\Gamma)$,
is defined as
\[
  h_{top}(\Gamma)\,:=\,\sup_{\eps>0}\,\limsup_{\nu\to \infty}\frac{1}{\nu}
  \log\big(\max\{\sharp\Fam : \Fam \text{ is an $(\eps, \nu)$-separated family}\}\big).
\]
\end{definition}

\begin{remark}
Since the manifold $X$ is compact, for any $\eps>0$ and $\nu\in \Z_+$, any
$(\eps, \nu)$-separated family in the above definition is finite. Furthermore, it is
routine to verify (see \cite[Section~4]{dinhSibony:ubtemc08}) that $h_{top}(\Gamma)$
does not depend on the choice of the metric $d$ for defining
$(\eps, \nu)$-separatedness.
\end{remark}

Comparing the above definition with Definition~\ref{D:Bowen_entropy}, we see that
Definition~\ref{D:entropy_D-S} is closely related to Bowen's definition of the topological entropy
of maps.
\smallskip

\begin{definition}\label{D:rat_semigroup}
A \emph{rational semigroup} is a semigroup, with composition of maps as the semigroup operation,
whose elements are surjective holomorphic self-maps of $\pro^n$ for some $n\in \Z_+$.
\end{definition}

\begin{remark}
Note that, despite the word ``rational'' in Definition~\ref{D:rat_semigroup}, the
elements of a rational semigroup on $\pro^n$, $n\geq 2$, do \textbf{not} possess indeterminacies.
We have some results about classical rational semigroups (i.e., defined on
$(\C\cup \{\infty\})\cong \pro^1$), and do not want to coin new terminology for theorems that
also hold true for higher-dimensional analogues of the latter semigroups. This is the reason for
the term ``rational semigroups'' introduced in Definition~\ref{D:rat_semigroup}. 
\end{remark}  

There is a very natural connection between \emph{finitely} generated rational semigroups and
holomorphic correspondences. Since this association defines the holomorphic correspondences
for which we shall make exact entropy computations, let us state it formally. This
association makes sense in greater generality, and not just for rational semigroups. 

\begin{definition}\label{D:assoc}
Let $X$ be a compact, connected complex manifold and let $S$ be a finitely generated semigroup
consisting of surjective holomorphic self-maps of $X$. Let $\gen = \{f_1,\dots, f_N\}$ be a set
of generators of $S$.
We call the holomorphic correspondence
\begin{equation}\label{E:assoc}
  \Gamma_{\gen} := \sum\nolimits_{1\leq j\leq N}{\sf graph}(f_j)
\end{equation}
on $X$ the \emph{holomorphic correspondence associated with $(S, \gen)$}.
\end{definition}

As has been observed earlier\,---\,see \cite{ghysLangevinWalczak:egf88} 
by Ghys, Langevin and Walczak
for the case of pseudogroups, or \cite{friedland:egsg96}\,---\,the entropy
of a finitely-generated semigroup requires the specification of a set of generators. Thus, for $X$ and $S$
as in Definition~\ref{D:assoc}, and for a choice $\gen$ of a set of generators of $S$, we formally set:
\[
  h_F(S, \gen)\,:=\,h_{F}\big(|\Gamma_{\gen}|\big) \quad\text{and}
  \quad h_{top}(S, \gen)\,:=\,h_{top}(\Gamma_{\gen}).
\]
We must mention that with $S$ as above, the definition of 
Friedland's entropy for $S$ in \cite{friedland:egsg96}, denoted by $h_F(S)$, is independent of the
choice of the set of generators. Its definition is
\[
  h_F(S)\,:=\,\inf\big\{h_F(S, \gen) :\ \text{$\gen$ is a finite set of generators of $S$}\big\}.
\]
Given this definition, it is essential to first understand the quantity $h_F(S, \gen)$ for a given
set of generators $\gen$\,---\,which is why we focus on $h_F(S, \gen)$ and $h_{top}(S, \gen)$
in this work.
\smallskip

Recall that if $f: \pro^n\lrarw \pro^n$
is holomorphic, then using homogeneous coordinates, we have:
\[
 f([z_0: z_1:\dots : z_n])\,=\,\big[\varf_0(z_0, z_1,\dots, z_n): \varf_1(z_0, z_1,\dots, z_n):
 \dots : \varf_n(z_0, z_1,\dots, z_n)\big],
\]
where there exists a number $d_1(f)\in \Z_+$, and
$\varf_0, \varf_1,\dots, \varf_n$ are homogeneous polynomials of degree $d_1(f)$ such that
$\cap_{i=0}^n\,\varf_i^{-1}\{0\} = \{0\}$. With this definition, we can state our first result. 

\begin{theorem}\label{T:main_D-S}
Let $S$ be a finitely generated rational semigroup on $\pro^n$ for some $n\in \Z_+$. Let
$\gen = \{f_1,\dots, f_N\}$ be a set of generators of $S$. Then
\[
  h_{top}(S, \gen)\,=\,\log\left(\,\sum\nolimits_{1\leq j\leq N} d_1(f_j)^n\right).
\]
\end{theorem}

We can say a lot more than Theorem~\ref{T:main_D-S}. The latter is a consequence of a more \textbf{general}
theorem, which provides bounds from above and below on $h_{top}(S, \gen)$ in a more general
context\,---\,see Theorem~\ref{T:general_bounds} below. 
Here, however: as the notion of a rational semigroup first arose in the area of complex dynamics in one
dimension (see \cite{hinkkanenMartin:dsrf96} by Hinkkanen and Martin)\,---\,and to foreshadow
Theorem~\ref{T:F_entr}\,---\,we state  the following special case:

\begin{corollary}\label{C:one-dim_D-S}
Let $S$ be a finitely generated rational semigroup on $\pro^1$, and let
$\gen = \{f_1,\dots, f_N\}$ be a set of generators of $S$. Then
\[
  h_{top}(S, \gen)\,=\,\log\left(\,\sum\nolimits_{1\leq j\leq N} \deg(f_j)\right).
\]  
\end{corollary}
 
We now turn to Friedland's entropy. Although it makes sense in a much more
general setting, $h_F$ turns out to be harder to compute. This is
because, among other reasons, notions that approximate
concepts such as irreducible components, etc., are much less well-structured outside the complex-analytic
setting, and do not feature in Definition~\ref{D:entropy_F}. In the complex-analytic setting, 
this absence leads to two difficulties
that one can point to (with $X$ here as in Definition~\ref{D:holCorr}): 
\renewcommand{\theenumi}{{\emph{\roman{enumi}}}}
\begin{enumerate}[leftmargin=27pt]
  \item\label{I:on_mult} A holomorphic correspondence $\Gamma$ on $X$
  can be iterated. If $d_{top}(\Gamma)$ denotes the topological degree of $\Gamma$ (see
  Section~\ref{S:prelim_complex}), then one has the identity $d_{top}(\Gamma^{\circ \nu}) =
  d_{top}(\Gamma)^{\nu}$ for any $\nu\in \Z_+$. The analogous identity for the $\nu$-fold iterate of the 
  \textbf{relation} $|\Gamma|$\,---\,which is relevant to the entropy $h_F$\,---\,is not true in
  general. This vitiates computations of $h_F$.
  \vspace{0.8mm}
  
  \item If $X$ is K{\"a}hler, then either type of entropy is dominated by the quantity
  ${\rm lov}(\Gamma)$ (see
  Section~\ref{S:prelim_complex})\,---\,which results from a technique of Gromov \cite{gromov:ehm03}.
  For similar reasons as in \eqref{I:on_mult}, ${\rm lov}(\Gamma)$ turns out not to be the 
  best upper bound for $h_F(\Gamma)$ even for $\Gamma=\Gamma_{\gen}$.
\end{enumerate}
We shall elaborate upon these points and discuss further the relationship between
the above notions of entropy in Section~\ref{SS:compare} below. For the moment, we note that 
rather few examples of \textbf{exact} computations of $h_F$ (not necessarily in the
holomorphic category) are known: 
see, for instance, \cite[Section~5]{friedland:egsg96} and
\cite{gellerPollicott:eZ2afeg} by Geller and Pollicott. But, as indicated above, computing $h_F$
is inherently hard. However, certain lower bounds for $h_F(S, \gen)$, $\gen$ finite, are almost
immediate: $h_F(S, \gen)\geq h_F(\langle f: f\in \gen^\prime\rangle, \gen^\prime)$ for any
$\varnothing\neq \gen^\prime\varsubsetneq \gen$ (see Section~\ref{S:main_F} for details).
In contrast to this, for $S$ a rational
semigroup on $\pro^1$, we shall establish a lower bound for $h_F(S, \gen)$ that takes into consideration
each of the generators in $\gen$. This is our Theorem~\ref{T:F_entr}. Since it requires some notation,
we present it in Section~\ref{S:main_F}. This theorem, in turn, relies on
our central proposition of Section~\ref{S:notn_top_ent}, which might be of independent interest. The proof
of Theorem~\ref{T:main_D-S} (from which Corollary~\ref{C:one-dim_D-S} is immediate) is presented in
Section~\ref{S:main_D-S}.

\subsection{A comparison of $\boldsymbol{h_{top}}$ and $\boldsymbol{h_{F}}$}\label{SS:compare}
This section presents a bit more discussion on the two notions of entropy defined above. The material
in the sections below does not depend on this discussion (and readers interested in the proofs can skip
this on the first reading). The first point of contrast involves this question: given a correspondence
$\Gamma$ with the presentation \eqref{E:stdForm}, is $h_{top}(\Gamma)=h_{F}(\Gamma)$ if the multiplicity
of each $\Gamma_j$ is $1$? The answer to this is ``No,'' in general. To understand this answer, let us consider
one of the roles these multiplicities serve. Consider a finitely
generated semigroup $S$ and a set of generators $\gen:=\{f_1,\dots,f_N\}$. If $S$ is not freely generated
and suppose, for some $\nu\in \Z_+\setminus\{1\}$, there exists a relation
\begin{equation}\label{E:rels_n}
 f_{i_{\nu}}\circ\dots \circ f_{i_1}\,=\,f_{j_{\nu}}\circ\dots \circ f_{j_1}\,=:\Phi,
\end{equation}
where $(i_1,\dots, i_{\nu})\neq (j_1,\dots, j_{\nu})$, then  
the irreducible component ${\sf graph}(\Phi)$ occurs with multiplicity at least $2$ in
the correspondence $\Gamma_{\gen}^{\circ\nu}$. Given such a semigroup and a choice, $\gen$, of a set
of generators, it is usually hard to count all the relations of the form \eqref{E:rels_n} as $\nu\to \infty$. This,
hopefully, reveals explicitly one of the reasons why $h_{F}(\Gamma)$ is difficult to compute in many cases.
For any holomorphic correspondence, a coarse manifestation of the issue of multiplicities\,---\,discussed here for
the correspondence $\Gamma_{\gen}$ associated with a rational semigroup\,---\,is the phenomenon mentioned
in \eqref{I:on_mult} above. Now, to return to the question above, the following example illustrates our answer.
Let $S$ be the rational semigroup generated by $\gen$, let $\gen_1=\{f_1, f_2\}$, where $f_1$ and $f_2$ are
distinct loxodromic M{\"o}bius transformations such that $f_1$ and $f_2$ have the same attracting fixed point and
the same repelling fixed point. Then, while both the irreducible components of the correspondence $\Gamma_{\gen}$
(as defined in \eqref{E:assoc}) occur with multiplicity $1$, we have
\begin{align*}
  h_{F}(\Gamma_{\gen})\,&=:\,h_{F}(S, \gen)\,=\,0, && (\,\text{see \cite[Lemma~5.2]{friedland:egsg96}}\,) \\
  h_{top}(\Gamma_{\gen})\,&=:\,h_{top}(S, \gen)\,=\,\log(2).
  && (\,\text{by Corollary~\ref{C:one-dim_D-S}}\,) 
\end{align*}
One relation between the two notions of entropy that holds universally is as follows: with $X$ as above and for
any holomorphic correspondence $\Gamma$ on $X$, $h_{F}(\Gamma)\leq h_{top}(\Gamma)$. This is 
immediate from Proposition~\ref{P:ent_ed} below.
\smallskip

Now, $h_{top}$ has some features that may seem anomalous. E.g., with $X$ as above, if $\Gamma$ is an
irreducible complex-analytic subvariety of $X\times X$ with the properties stated in Definition~\ref{D:holCorr},
then $h_{top}(2\Gamma)\neq h_{top}(\Gamma)$. When $X=\pro^n$ and $\Gamma$ is the graph of a
non-constant holomorphic map $f: \pro^n\lrarw \pro^n$, then it follows from the \textbf{proof} of 
Theorem~\ref{T:main_D-S} that $h_{top}(2\Gamma) = \log(2)+h_{top}(\Gamma)$. But this value
isn't necessarily anomalous in a context where one must consider $2\Gamma$ (instead of $\Gamma$).
As a lot of the formalism of this paper is that of general holomorphic correspondences, one may ask
why one cares for the entropy of holomorphic correspondences. It would be natural to study the entropy of
these objects if one cares about the iterative dynamics of correspondences. Correspondences serve as
a common framework\,---\,as the theorems above and the discussion in Section~0 of \cite{friedland:egsg96}
testify\,---\,for a number of dynamical systems of interest. The iteration of true holomorphic correspondences
on $\pro^1$ (i.e., \textbf{not} maps) are of interest too: they realise matings between certain rational maps
and certain classes of Kleinian groups: see \cite{bullFrei:hcmC-lmHb05, bullettPenrose:mqmmg94}, for instance.
\medskip

\section{Complex geometry preliminaries}\label{S:prelim_complex}

This section is devoted to a discussion of terminology from geometry appearing in
Section~\ref{S:intro} whose definitions had been deferred, and to stating a result that constitutes
one part of the proofs of Theorem~\ref{T:main_D-S} and Corollary~\ref{C:one-dim_D-S}.
\smallskip

We first begin with a discussion of the composition of two holomorphic correspondences. Since
one needs to understand this \textbf{only} to define a certain finite sequence of numbers associated
to a correspondence $\Gamma$, we shall be brief. We refer the reader to 
\cite[Section~3]{dinhSibony:ubtemc08} for details (with a note to those unfamiliar with holomorphic
correspondences that the footnote to \cite[Section~3]{dinhSibony:ubtemc08} is irrelevant in the
case of holomorphic correspondences). We focus on two points that are relevant to
this article (in what follows, $\Gamma^{\circ\nu}$ will denote the $\ordi{\nu}$ iterated composition
of $\Gamma$): 
\renewcommand{\theenumi}{{\emph{\roman{enumi}}}}
\begin{enumerate}[leftmargin=27pt]
  \item With $X$ as in Definition~\ref{D:holCorr}, consider two holomorphic correspondences
  \[
    \Gamma^1\,=\,\Add{1\leq j\leq M_1}\!\IVar_{1,\,j} \qquad\text{and}
    \qquad \Gamma^2\,=\,\Add{1\leq k\leq M_2}\!\IVar_{2,\,k}
  \]
  on $X$, written in accordance with the presentation \eqref{E:stdForm_mult}. The support
  of $\Gamma^2\circ \Gamma^1$ is just the classical composition of
  $|\Gamma^2|$ with $|\Gamma^1|$ as \textbf{relations}. Let us denote the latter composition
  by $\star$. If $Y_{s,\,jk}$, $s=1,\dots, M(j, k)$, are the distinct irreducible components of
  $|\IVar_{2,\,k}|\star|\IVar_{1,\,j}|$, then let
  \begin{align*}
    \eta_{s,\,jk}\,:=\,&\text{the number of $y$'s, for a generic $(x,z)\in Y_{s,\,jk}$, such that} \\ 
  					&\text{$(x, y)\in \IVar_{1,\,j}$ \& $(y, z)\in \IVar_{2,\,k}$.}
  \end{align*}
  Then, the definition in \cite[Section~3]{dinhSibony:ubtemc08} results in the formula:
  \[
    \Gamma^2\circ\Gamma^1\,:= \ \sum_{1\leq j\leq M_1}\;\sum_{1\leq k\leq M_2}\;
    								\sum_{1\leq s\leq M(j,\,k)}
  								\eta_{s,\,jk}Y_{s,\,jk}.
  \]
  \item For the semigroup $S$, a choice of a set of generators $\gen$, and the correspondence
  $\Gamma_{\gen}$ introduced in Definition~\ref{D:assoc}, we have
  \[
    \Gamma_{\gen}^{\circ 2}\,:=\,\Gamma_{\gen}
    \circ \Gamma_{\gen}\,=\,\sum\nolimits_{1\leq j,\,k\leq N}{\sf graph}(g_j\circ g_k).
  \]
\end{enumerate}
Observe that if $S$ is not a free semigroup and if, for instance, there exists a relation of the form
$g_{j_1}\circ g_{k_1} = g_{j_2}\circ g_{k_2}$ for $(j_1, k_1)\neq (j_2, k_2)$, then
the irreducible variety ${\sf graph}(g_{j_1}\circ g_{k_1})$ would occur with multiplicity
at least $2$. Observations such as the one above are the reason why the data
$(m_1,\dots, m_N)$ in \eqref{E:stdForm} are essential in defining a holomorphic correspondence.
\smallskip

One can pull back certain types of currents by a holomorphic correspondence\,---\,see
\cite[Section~3.1]{dinhSibony:dvtma06}.
The \emph{formal} prescription for the pullback (which we denote by $F_{\Gamma}^*$)
of any current $T$ of bidegree $(p,p)$, $p=0, 1,\dots, n$ (recall that $\dim_{\C}(X) = n$) is:
\begin{equation}\label{E:prescrip}
  F_{\Gamma}^*(T)\,:=\,(\pi_1)_*\left(\pi_2^*(T)\wedge [\Gamma]\right)
\end{equation}
whenever the intersection $\pi_2^*(T)\wedge [\Gamma]$ makes sense. Here,
$\Gamma$ detemines a current of bidimension $(n,n)$ given by the currents of integration defined by its
constituent subvarieties\,---\,which we denote by $[\Gamma]$. Recall that if $X_1$ and $X_2$
are two compact, connected complex manifolds of dimensions $n_1$ and $n_2$, respectively, $\pi: X_1\lrarw X_2$ is
a holomorphic map, and $T$ is a current on $X_1$ of bidegree $(p,p)$,
$\max(n_1-n_2, 0)\leq p\leq n_1$, then
the push-forward of $T$ by $\pi$ is given by
\[
  \langle \pi_*{T}, \varphi\rangle\,:=\,\langle T, \pi^*\varphi\rangle
  \quad \text{$\forall (n_1-p, n_1-p)$-forms $\varphi$ on $X_2$},
\]
whereby $\pi_*{T}$ is a current of bidegree $(n_2-n_1+p, n_2-n_1+p)$ on $X_2$ (with the
understanding that if $0\leq p< n_1-n_2$, then 
$\pi_*{T}\equiv 0$ for any $(p,p)$-current $T$). With $X_1$, $X_2$ and $\pi$ as before, the pullback of a current 
by $\pi$ is somewhat non-standard. To begin with, one can define a pullback if $\pi$ is a submersion onto $X_2$.
If $\pi$ is not a submersion, then the pullback is defined
for special classes of $(p,p)$-currents on $X_2$. We refer the reader to \cite[Section~2.4]{dinhSibony:dvtma06}
for the definition of the pullback of $T$ by $\pi$ in these various cases. In \eqref{E:prescrip}, the map $\pi_2$
is a submersion and is simple enough that we have
\[
  \langle \pi_2^*(T),\,\varphi\rangle\,:=\,\Big\langle T,\,\int\nolimits_{x\in X}\!\varphi(x, \bcdot)\Big\rangle
  \quad \text{$\forall (2n-p, 2n-p)$-forms $\varphi$ on $X\times X$},
\]
for a $(p,p)$-current $T$, $0\leq p\leq n$.  
\smallskip
    
We consider an example where the intersection of currents in \eqref{E:prescrip} makes sense. Any smooth $(p,p)$-form
$\Theta$ on $X$, $p=0, 1,\dots, n$, can be pulled back by $\Gamma$ to give a $(p,p)$-current (equivalenty,
a current of bidimension $(n-p, n-p)$) as follows:
\[
  \langle F_{\Gamma}^*(\Theta),\,\varphi\rangle\,:=\,\sum_{j=1}^N\,m_j\!\int_{{\sf reg}(\Gamma_j)}\!\!
  								\big(\left. \pi_2\right|_{\Gamma_j}\big)^{\!*}\Theta\wedge
  								\big(\left. \pi_1\right|_{\Gamma_j}\big)^{\!*}\varphi
  								\quad \text{$\forall (n-p, n-p)$-forms $\varphi$,}
\]
using the presentation \eqref{E:stdForm} for $\Gamma$. Now suppose $(X, \omega)$ is a K{\"a}hler manifold
and let $\omega_X$ denote the normalisation of $\omega$ so that $\int_X \omega_X^n = 1$.  
For $p=0, 1,\dots, n$, we define the
\emph{$\ordi{p}$ intermediate degree of $\Gamma$} by
\[
  \lambda_p(\Gamma)\,:=\,\langle F_{\Gamma}^*(\omega_X^p),\,\omega_X^{n-p}\rangle.
\]
It is well known that for each $p$, $\lambda_p$ is sub-multiplicative with respect to composition. Thus,
the limit on the right-hand side below
\begin{equation}\label{E:d_deg}
  d_p(\Gamma)\,:=\,\lim_{\nu\to \infty}\lambda_p(\Gamma^{\circ \nu})^{1/\nu},
  \quad p = 0, 1,\dots, n,
\end{equation}
exists. The number $d_p(\Gamma)$ is called the
\emph{$\ordi{p}$ dynamical degree of $\Gamma$}. Since the limit on the right-hand side of
\eqref{E:d_deg} exists, $d_p(\Gamma^{\circ k}) = d_p(\Gamma)^k$, $p = 0, 1,\dots, n$, for
every $k\in \Z_+$.
\smallskip

With these definitions, we can state a result that we shall need in proving Theorem~\ref{T:main_D-S}
and Corollary~\ref{C:one-dim_D-S}.

\begin{result}[paraphrasing of {\cite[Theorem~1.1]{dinhSibony:ubtemc08}}]\label{R:ent_u-bd}
Let $(X, \omega)$ be a compact K{\"a}hler manifold of dimension $n$ and let $\Gamma$ be a
holomorphic correspondence on $X$. Then
\[
  h_{top}(\Gamma)\,\leq\,\max_{0\leq p\leq n}\log d_p(\Gamma).
\]
\end{result}

We ought to mention that Dinh--Sibony establish the above bound on $h_{top}$ for the more
general class of \emph{meromorphic} correspondences. Furthermore, this bound is actually
obtained\,---\,adapting a technique of Gromov \cite{gromov:ehm03}\,---\,by computing
the value of ${\rm lov}(\Gamma)$, which dominates $h_{top}(\Gamma)$. Roughly
speaking, ${\rm lov}(\Gamma)$ is the asymptotic rate of logarithmic growth (relative to $\nu$) of the
volume of the space of all $\nu$-orbits.   
\smallskip

\section{Notation and essential propositions on topological entropy}\label{S:notn_top_ent}
We begin by fixing some notation that will be needed for the propositions in this section and in subsequent
sections. The objects introduced here will pertain to a \emph{general} holomorphic correspondence
$\Gamma$, and our notation will be with reference to the presentation \eqref{E:stdForm_mult} of $\Gamma$.
\smallskip

We begin by introducing an object similar to $\Gamma^\infty$ of Definition~\ref{D:entropy_F}. The parameter
$M$ has the same meaning in the following definition as in \eqref{E:stdForm_mult}:
\[
  \pthsp{\Gamma}\,:=\,\{(x_0, x_1, x_2,\dots; \alpha_1, \alpha_2,\dots)\in
  X^\N\times \{1,\dots,M\}^{\Z_+}: (x_{\nu-1}, x_{\nu})\in \IVar_{\alpha_\nu} \; \ \forall \nu\in \Z_+\}.
\]
This space is endowed
with the topology that it inherits from $X^{\N}\times \{1,\dots, M\}^{\Z_+}$ endowed with the product
topology. We will denote by $\pth{\Gamma}_{\nu}$ the space of all $\nu$-orbits, i.e.,
\[
  \pth{\Gamma}_{\nu}\,:=\,\{(x_0,\dots x_{\nu}; \alpha_1,\dots, \alpha_{\nu})\in
  X^{\nu+1}\times \{1,\dots,M\}^{\nu}: (x_{j-1}, x_{j})\in \IVar_{\alpha_j}, \; \ 1\leq j\leq \nu\}.
\]
The above is endowed with the relative topology that it inherits from
$X^{\nu+1}\times \{1,\dots,M\}^{\nu}$.
\smallskip

We shall need the following maps. By a mild abuse of notation, we shall denote by $\iti$ either the
map $\iti: \pthsp{\Gamma}\lrarw X^{\N}$ or the map $\iti: \pth{\Gamma}_{\nu}\lrarw X^{\nu+1}$
that maps the relevant orbit of an iteration under $\Gamma$ to the itinerary of points in $X$ along that
orbit. In other words:
\begin{align*}
  \iti: \pthsp{\Gamma}\,\ni (x_0, x_1,\dots; \alpha_1,\dots)\,&\mapsto (x_0, x_1,\dots) \; \;\text{or} \\
  \iti: \pth{\Gamma}_{\nu}\,\ni (x_0, x_1,\dots, x_{\nu};
  \alpha_1,\dots, \alpha_{\nu})\,&\mapsto (x_0, x_1,\dots, x_{\nu}) \; \; \text{respectively,}  
\end{align*}
where the precise definition of $\iti$ will be \textbf{obvious} from the context. By a similar abuse of notation,
we shall denote by $\wrd$ either the
map $\wrd: \pthsp{\Gamma}\lrarw \{1,\dots, M\}^{\Z_+}$ or the map
$\wrd: \pth{\Gamma}_{\nu}\lrarw \{1,\dots, M\}^{\nu}$, defined by
\begin{align*}
  \wrd: \pthsp{\Gamma}\,\ni (x_0, x_1,\dots; \alpha_1,\dots)\,&\mapsto (\alpha_1,\dots) \quad\text{or} \\
  \wrd: \pth{\Gamma}_{\nu}\,\ni (x_0, x_1,\dots, x_{\nu};
  \alpha_1,\dots, \alpha_{\nu})\,&\mapsto (\alpha_1,\dots, \alpha_{\nu}) \; \; \text{respectively.}  
\end{align*}
Lastly, $\sta: \pth{\Gamma}_{\nu}\lrarw X$ will denote the map
$(x_0, x_1,\dots, x_{\nu};\alpha_1,\dots, \alpha_{\nu})\mapsto x_0$. The need to
consider sets of finite and infinite orbits (in ways similar to how
the above objects are used here) has arisen earlier in the literature. E.g.,
formalisms similar to those given here for correspondences are seen in studying surjective
holomorphic self-maps of $\pro^2$ due to the need to consider different prehistories (i.e.,
backward orbits) of a point $x_0\in \pro^2$\,---\,see, e.g., 
\cite{mihaUrba:ttpahdscv04} and \cite{fornaessMiha:emsshmPP13}.%
\smallskip

We shall also need a standard result in elementary topology (also see
Remark~\ref{Rem:f_shift}). We shall abbreviate $\sigma^{\circ j}$ as 
$\sigma^j$.

\begin{lemma}\label{L:basic}
Let $(Y, \met)$ be a compact metric space and let $\wht{\met}$ denote the metric
\[
  \wht{\met}(\wht{x}, \wht{y})\,:=\,\sup_{n\in \N}\frac{\met(x_n, y_n)}{2^n},
\]
$\wht{x}:= (x_0, x_1, x_2,\dots)$ and  $\wht{y}:= (y_0, y_1, y_2,\dots)$, which metrises
the product topology on $Y^\N$. Let $\sigma: (x_0, x_1, x_2,\dots) \longmapsto
(x_1, x_2, x_3,\dots)$ be the left-shift on $Y^\N$. Then:
\[
  \max_{0\leq j\leq n}\wht{\met}\big(\sigma^{j}(\wht{x}),
  \sigma^{j}(\wht{y})\big)\,=\,\sup_{j\in \N}\frac{\met(x_j, y_j)}{2^{(j-n)+}},
\]
where $(j-n)+ := \max(j-n, 0)$.
\end{lemma}

Before stating the principal result of this section, we provide a couple of clarifications.
In what follows, if $Y$ is a compact metric space and $f: Y\lrarw Y$ is a continuous map, then
$h(f)$ will denote its topological entropy in the sense of Bowen. We shall define this in the
setting of compact metric spaces, along with a remark on the broader concept introduced by 
Bowen in \cite{bowen:egehp71}. Since the phrase ``$(\eps, \nu)$-separated'' appears in 
Definition~\ref{D:entropy_D-S} in a (slightly) different context, we shall\,---\,to avoid
confusion\,---\,use wording that is slightly different from that in \cite{bowen:egehp71}.

\begin{definition}\label{D:Bowen_entropy}
Let $(Y, \met)$ be a compact metric space and let $f: Y\lrarw Y$ be a continuous map.
A set $\Fam$ of orbits of $f$ of duration $n$, $n\in \Z_+$, is said to be an \emph{$\eps$-separated
set of orbits of $f$ of duration $n$}, $\eps>0$, if for all pairs of distinct orbits
\[
  (x_0, x_1,\dots, x_{n}) \quad\text{and}
    \quad (y_0, y_1,\dots, y_{n})
\]
of $\Fam$ (i.e., $x_j=f^j(x_0)$, $y_j=f^j(y_0)$, $j=1,\dots,n$)
\[
  \met(x_j, y_j)\,>\,\eps \text{ for some $j=0,1,\dots, n$}.
\]
Let $M(\eps, n)$ denote the greatest possible cardinality of any $\eps$-separated
set $\Fam$ of orbits of $f$ of duration $n$. Then, the \emph{topological entropy of
$f$}, denoted by $h(f)$, is defined as
\[
  h(f)\,:=\,\sup_{\eps>0}\,\limsup_{n\to \infty}\frac{1}{n}\log M(\eps, n).
\]
\end{definition}

\begin{remark}\label{Rem:more_Bowen}
In \cite{bowen:egehp71}, Bowen gives a definition of entropy that does not require $Y$ to be compact. In that
case, $f$ must be uniformly continuous relative to the metric on $Y$, and the value of Bowen's entropy of
$f$ depends on this metric. But when $Y$ is compact, Bowen's entropy is independent of the
metric, provided it metrises the topology on $Y$. This is the special case of the broader framework in
\cite{bowen:egehp71} that we focus on in Definition~\ref{D:Bowen_entropy}. 
\end{remark}

\begin{remark}\label{Rem:f_shift}
In Definition~\ref{D:entropy_F}, where $h_F(\Gamma)$ is defined as $h(\sigma)$
for the shift map $\sigma: \Gamma^\infty\lrarw \Gamma^\infty$, we now see that the
latter is given by Definition~\ref{D:Bowen_entropy} with $Y=\Gamma^\infty$ and
$f=\sigma$. This, in fact, is the motivation for Lemma~\ref{L:basic}.
\end{remark}

We now state and prove a result that will be needed in the proof of Theorem~\ref{T:F_entr}.
This result is hinted at in \cite[Section~4]{dinhSibony:ubtemc08}. However:
\begin{itemize}[leftmargin=24pt]
  \item It is unclear if Proposition~\ref{P:ent_ed} follows, as alluded to in \cite{dinhSibony:ubtemc08},
  from the conjugacy invariance of topological
  entropy (which applies to pairs of maps).
  \vspace{0.8mm}
  
  \item For $\Gamma$ as in Definition~\ref{D:holCorr}, if the topological degree of
  $\big(\left.\pi_1\right|_{\Gamma_j}\big)\geq 2$ for any $j\in \{1,\dots, N\}$, then it is unclear
  whether $\Gamma$ can at all be conjugated to a shift on $\pthsp{\Gamma}$.
\end{itemize} 
While we shall apply Proposition~\ref{P:ent_ed} only to the holomorphic correspondence
$\Gamma_{\gen}$ in Section~\ref{S:main_F}, it holds true for general holomorphic
correspondences. It may thus be of independent interest. In view of the two points
above, it seems worthwhile to state and give a \emph{direct} proof of

\begin{proposition}\label{P:ent_ed}
Let $X$ and $\Gamma$ be as in Definition~\ref{D:holCorr}. For $\eps > 0$ and $\nu\in \Z_+$,
let
\begin{align*}
  N(\eps, \nu)\,:=&\,\text{the cardinality of any $(\eps, \nu)$-separated family of $\nu$-orbits, in the} \\
  &\,\text{sense of Definition~\ref{D:entropy_D-S}, having the greatest possible cardinality.}
\end{align*}
Let $\sh$ denote the restriction of the shift map
\[
  \sigma: (x_0, x_1, x_2,\dots; \alpha_1, \alpha_2,\dots)\mapsto
  (x_1, x_2, x_3,\dots; \alpha_2, \alpha_3,\dots)
\]
to $\pthsp{\Gamma}$. Then,
\[
   h_{top}(\Gamma)\,:=\,\sup_{\eps>0}\,\limsup_{\nu\to \infty}\frac{1}{\nu}
   \log{N(\eps, \nu)}\,=\,h(\sh),
\]
where $h(\sh)$ is the entropy, in the sense of Bowen, of the continuous map $\sh: \pthsp{\Gamma}
\lrarw \pthsp{\Gamma}$.
\end{proposition}
\begin{proof}
Let us fix a metric $d$ on the complex manifold $X$ that is compatible with the manifold topology.
We choose the metric
\[
  \Delta\big( (x_0, x_1,\dots ; \alpha_1,\dots), 
  (y_0, y_1,\dots ; \beta_1,\dots) \big)
  :=\,\max\left[\,\sup_{\nu\in \N}\frac{d(x_{\nu}, y_{\nu})}{2^{\nu}}, \ 
  \sup_{\nu\in \N}\frac{\kd(\alpha_{\nu+1}, \beta_{\nu+1})}{2^{\nu}}\,\right]
\]
(where $\kd$ denotes the $0$-$1$ metric on the symbols $\{1,\dots, M\}$) which metrises the topology
on $\pthsp{\Gamma}$. Since we must
show that $h_{top}(\Gamma)$ equals the Bowen entropy of $\sh$, we introduce, for $\eps > 0$
and $\nu\in \Z_+$, the set
\begin{align*}
  M(\eps, \nu)\,:=&\,\text{the cardinality of any $\eps$-separated set of orbits of $\sh$} \\
  &\,\text{of duration $\nu$ having the greatest possible cardinality.}
\end{align*}
Recall that any two orbits,
\begin{equation}\label{E:notation}
  \mathscr{O}_1\,:=\,(x_0, x_1,\dots ; \alpha_1,\dots) \quad\text{and}
  \quad \mathscr{O}_2\,:=\,(y_0, y_1,\dots ; \beta_1,\dots),
\end{equation}
belonging to any of the sets referred to in the definition of $M(\eps, \nu)$ satisfy
\begin{equation}\label{E:sep_nu}
  \max_{0\leq j\leq \nu}\Delta(\sh^j\big(\mathscr{O}_1), \sh^j(\mathscr{O}_2)\big)\,>\,\eps.
\end{equation}
Fix a $\nu\in \Z_+$. It suffices to consider $\eps\in (0, 1)$.
\smallskip

Let $S(\eps, \nu)\subset \pth{\Gamma}_{\nu}$ be an $(\eps, \nu)$-separated family, in
the sense of Definition~\ref{D:entropy_D-S}, such that $\sharp{S(\eps, \nu)} = N(\eps, \nu)$.
For each $\nu$-orbit $\bsy{x}\in S(\eps, \nu)$,
let us pick a $(x_0, x_1, x_2,\dots ; \alpha_1, \alpha_2\dots)\in \pthsp{\Gamma}$ such that
\[
  (x_0, x_1,\dots, x_\nu; \alpha_1,\dots, \alpha_\nu)\,=\,\bsy{x},
\]
and \textbf{fix} it. Call the latter infinite orbit $\widetilde{\bsy{x}}$. Let us
consider two distinct $\nu$-orbits
\[
  \bsy{x}\,:=\,(x_0, x_1,\dots, x_\nu; \alpha_1,\dots, \alpha_\nu) \quad\text{and}
  \quad\bsy{y}\,:=\,(y_0, y_1,\dots, y_\nu; \beta_1,\dots, \beta_\nu)
\]
belonging to $S(\eps, \nu)$. We have two possibilities for the pair $\{\bsy{x}, \bsy{y}\}$:
\medskip

\noindent{\textbf{Case~1.} $\max_{\,0\leq j\leq \nu}d(x_j, y_j) > \eps$.}
\vspace{0.6mm}

\noindent{Then (with the meaning of $\widetilde{\bsy{y}}$ hopefully being clear) by the
definition of $\Delta$, and in view of Lemma~\ref{L:basic}, we have
\begin{equation}\label{E:eps1}
  \max_{0\leq j\leq \nu}\Delta\big(\sh^{j}(\widetilde{\bsy{x}}),
  \sh^{j}(\widetilde{\bsy{y}})\big)\,\geq\,\sup_{j\in \N}\frac{d(x_j, y_j)}{2^{(j-\nu)+}}\,>\,\eps.
\end{equation}}
\medskip

\noindent{\textbf{Case~2.} $\max_{\,0\leq j\leq \nu}d(x_j, y_j)\leq \eps$.}
\vspace{0.6mm}

\noindent{In this case, by \eqref{E:eps-nu_sep} there exists a $j^*$, with $1\leq j^*\leq \nu$,
such that $\alpha_{j^*}\neq \beta_{j^*}$. Therefore, in view of Lemma~\ref{L:basic}, we have
\begin{equation}\label{E:eps2}
  \max_{0\leq j\leq \nu}\Delta\big(\sh^{j}(\widetilde{\bsy{x}}),
  \sh^{j}(\widetilde{\bsy{y}})\big)\,\geq\,\sup_{j\in \N}
  \frac{\kd(\alpha_{j+1}, \beta_{j+1})}{2^{(j-\nu)+}}\,\geq\,1\,>\,\eps.
\end{equation}}
\medskip

From \eqref{E:eps1} and \eqref{E:eps2} it follows that the set
$\{\widetilde{\bsy{x}}\in \pth{\Gamma}: \bsy{x}\in S(\eps, \nu)\}$ is an
$\eps$-separated set of orbits of $\sh$ in the sense of \eqref{E:sep_nu}.
Since the latter set has cardinality $N(\eps, \nu)$, we get:
\begin{equation}\label{E:sizes1}
  M(\eps, \nu)\,\geq\,N(\eps, \nu).
\end{equation}

Now let $\Sigma(\eps, \nu)\subset \pthsp{\Gamma}$ be an $\eps$-separated set of orbits in the
sense of \eqref{E:sep_nu} such that $\sharp\Sigma(\eps, \nu) = M(\eps, \nu)$. Then, for
two distinct orbits $\mathscr{O}_1, \mathscr{O}_2\in \Sigma(\eps, \nu)$, we have (using the
notation in \eqref{E:notation}):
\begin{align}
  \frac{d(x_j, y_j)}{2^{(j-\nu)+}}\,&\leq\,\eps \quad \forall j\geq \nu+\log_2(1/\eps)+\log_2({\rm diam}(X)),
  \label{E:small1} \\
  \frac{\kd(\alpha_{j+1}, \beta_{j+1})}{2^{(j-\nu)+}}\,&\leq\,\eps \quad \forall j\geq \nu + \log_2(1/\eps), \label{E:small2}
\end{align}
where $\log_2(t) := \log(t)/\log(2) \ \forall t>0$. 
Given the definition of the metric $\Delta$, it is impossible for the quantities
\[
  \max_{0\leq j\leq \nu}\,\sup_{k\in \N}\frac{d\big(\pi_k\circ\iti(\sh^j(\mathscr{O}_1)),
  \pi_k\circ\iti(\sh^j(\mathscr{O}_2))\big)}{2^k},
\]
and
\[
  \max_{0\leq j\leq \nu}\,\sup_{k\in \Z_+}\frac{\kd\big(\pi_k\circ\wrd(\sh^j(\mathscr{O}_1)),
  \pi_k\circ\wrd(\sh^j(\mathscr{O}_2))\big)}{2^{k-1}}
\]
(where $\pi_k$ denotes the projection onto the $\ordi{k}$ factor) to \textbf{both} be less than or equal to $\eps$.
Thus, by \eqref{E:small1}, \eqref{E:small2} and Lemma~\ref{L:basic}, we have
\[
  d(x_j, y_j)\,>\,2^{(j-\nu)+}\eps \quad\text{or}
  \quad\kd(\alpha_{j+1}, \beta_{j+1})\,\neq\,0 \quad\text{for some $j: 0\leq j\leq C(\eps)+\nu$,}
\]
where $C(\eps)$ is the greatest integer that is strictly less than $\log_2(1/\eps) +\log_2({\rm diam}(X))$.
Hence
\[
  (x_0, x_1\dots, x_{C(\eps)+\nu}; \alpha_1,\dots, \alpha_{C(\eps)+\nu}) \quad\text{and}
 \quad(y_0, y_1\dots, y_{C(\eps)+\nu}; \beta_1,\dots, \beta_{C(\eps)+\nu})
\]
are $(\eps, C(\eps)+\nu)$-separated in the sense of Definition~\ref{D:entropy_D-S}.
Since this applies
to any pair of distinct $\mathscr{O}_1, \mathscr{O}_2\in \Sigma(\eps, \nu)$, we get
\[
  N(\eps, C(\eps) + \nu)\,\geq\,M(\eps, \nu).
\]
From this and \eqref{E:sizes1}, it follows that
\[
  N(\eps, \nu)\,\leq\,M(\eps, \nu)\,\leq\,N(\eps, C(\eps)+\nu).
\]
From the above, and from the definitions of the numbers $M(\eps, \nu)$ and $N(\eps, \nu)$,
the result is now immediate.
\end{proof}

We end this section revisiting Bowen's entropy. 
It will be needed in the proof of Theorem~\ref{T:F_entr}. To state it, we need
some terminology. Let $(Y,D)$ be as in Lemma~\ref{L:basic} and let $f: Y\lrarw Y$
be a continuous map. Let $K\subseteq Y$. Given $\eps>0$ and $n\in \Z_+$, a
subset $\Fam\subset Y$ is said to \emph{$(\eps, n)$-span $K$ with respect to $f$} if for
each $x\in K$ there exists a $y\in \Fam$ so that
\[
  D\big(f^{j}(x), f^{j}(y)\big)\,\leq\,\eps \quad \forall j=0,1,\dots,n-1.
\]
Let $r_n(\eps, K):= \inf\{\sharp\Fam : \Fam\subset Y
\text{ $(\eps, n)$-spans $K$}\}$. If $K$ is compact, then, clearly, $r_n(\eps, K)$
is finite for any $\eps>0$ and $n\in \Z_+$. Now, define:
\begin{equation}\label{E:f_K}
  h(f, K)\,:=\,\lim_{\eps\to 0^+}\,\limsup_{n\to \infty}\frac{1}{n}\log\big(r_n(\eps, K)\big).
\end{equation}
We must admit that, in the above discussion, we are omitting a considerable amount of
context. For instance, the quantity $h(f, K)$ is an ingredient in the definition of Bowen's
entropy, which\,---\,as mentioned in Remark~\ref{Rem:more_Bowen}\,---\,does
not require $Y$ to be compact. Before we state the result that we need, we must mention
that as $Y$ above is compact, $h(f, K)$ does not depend on the choice of $D$ (provided
it metrises the topology on $Y$): see the proof of \cite[Proposition~3]{bowen:egehp71}.

\begin{result}[Bowen, {\cite[Theorem~17]{bowen:egehp71}}]\label{R:bowen_quots}
Let $(Y_i, d_i)$, $i = 1,2$, be two compact metric spaces. Let $f_i : Y_i\lrarw Y_i$, $i = 1, 2$, be
continuous surjective maps. Let $\pi : Y_1\lrarw Y_2$ be a continuous surjective map such that
$\pi\circ f_1 = f_2\circ\pi$. Then
\[
  h(f_2)\,\leq\,h(f_1)\,\leq\,h(f_2) + \sup_{y\in Y_2}h(f_1,\,\pi^{-1}\{y\}).
\]  
\end{result}

\section{The proof of Theorem~\ref{T:main_D-S}}\label{S:main_D-S}
This section will chiefly be devoted to Theorem~\ref{T:general_bounds} below.
Theorem~\ref{T:main_D-S} would follow as its corollary. But before we can prove
Theorem~\ref{T:general_bounds}, we must present an auxiliary quantity and a
lemma. To do so, let $X$, $\Gamma$ and $d$ be as in Definition~\ref{D:entropy_D-S}.
For a given $\nu\in \Z_+$, \textbf{fix} a $\nu$-tuple $\wod :=
(\alpha_1,\dots, \alpha_{\nu})\in \{1,\dots, M\}^{\nu}$. We say that a family
$\Fam$ of $\nu$-orbits is $(\eps, \wod)$-separated if for all pairs of distinct elements
\[
  (x_0, x_1,\dots, x_{\nu}; \beta_1,\dots, \beta_{\nu}) \quad\text{and}
  \quad (y_0, y_1,\dots, y_{\nu}; \gamma_1,\dots, \gamma_{\nu})
\]
of $\Fam$, we have
\begin{itemize}[leftmargin=24pt]
  \item $(\beta_1,\dots, \beta_{\nu}) = \wod = (\gamma_1,\dots, \gamma_{\nu})$; and
  \item $\max_{0\leq j\leq \nu}d(x_j, y_j) > \eps$.
\end{itemize}

\begin{lemma}\label{L:sum-up}
Let $X$ and $\Gamma$ be as in Definition~\ref{D:entropy_D-S}.
For $\eps > 0$, $\nu\in \Z_+$ and $\wod\in \{1,\dots, M\}^{\nu}$, let
\begin{align*}
  n(\eps, \wod)\,:=&\,\text{the cardinality of any $(\eps, \wod)$-separated family of} \\
  &\,\text{$\nu$-orbits having the greatest possible cardinality.}
\end{align*}
Then,
\[
   h_{top}(\Gamma)\,=\,\sup_{\eps>0}\,\limsup_{\nu\to \infty}\frac{1}{\nu}
   \log\Bigg[\sum_{\wod\in \{1,\dots, M\}^{\nu}}\!\!n(\eps, \wod) \Bigg].
\]
\end{lemma}

\noindent{The proof of this lemma is extremely elementary.
But since it is vital to the proof of Theorem~\ref{T:general_bounds}, we provide the following}
\smallskip

\noindent{\emph{Outline of proof.} Fix an $\eps > 0$ and $\nu\in \Z_+$. For
$\wod\in \{1,\dots, M\}$, let $\Fam(\wod)$ be an $(\eps, \wod)$-separated
family such that $\sharp\Fam(\wod) = n(\eps, \wod)$. Write
\[
  \Fam\,:=\,\bigcup_{\wod\in \{1,\dots, M\}^{\nu}}\Fam(\wod).
\]
The lemma follows from the fact that $\Fam$ is an $(\eps, \nu)$-separated family
and that $\sharp\Fam = N(\eps, \nu)$\,---\,where $N(\eps, \nu)$ is as introduced
in Proposition~\ref{P:ent_ed}. Both these statements follow from the definitions and
the fact that if
\[
  (x_0, x_1,\dots, x_{\nu}; \alpha_1,\dots, \alpha_{\nu})\in
  \Fam\big((\alpha_1,\dots, \alpha_{\nu})\big), \quad
  (y_0, y_1,\dots, y_{\nu}; \beta_1,\dots, \beta_{\nu})\in
  \Fam\big((\beta_1,\dots, \beta_{\nu})\big),
\]
and $(\alpha_1,\dots, \alpha_{\nu})\neq (\beta_1,\dots, \beta_{\nu})$ then
these two $\nu$-orbits are $(\eps, \nu)$-separated in the sense of
Definition~\ref{D:entropy_D-S}. \hfill $\Box$}
\smallskip

We can now present the central result of this section. In what follows, $d_{top}$ will
denote the topological degree, while for a surjective holomorphic map
$f : X\lrarw X$, $X$ a compact K{\"a}hler manifold, $d_p(f)$ will
denote the $\ordi{p}$ dynamical degree of ${\sf graph}(f)$ (see
Section~\ref{S:prelim_complex}). We must also spell out what is meant by
the topological degree of a holomorphic correspondence. 
Let $X_1$ and $X_2$ be as in Definition~\ref{D:holCorr} and $\Gamma$ be a holomorphic correspondence
from $X_1$ to $X_2$. Representing $\Gamma$ as in
\eqref{E:stdForm}, it is classical that there is a Zariski-open
set $W\subset X_2$ and $d_j\in \Z_+$ such that $(\pi_2^{-1}(W)\cap \Gamma_j, W, \pi_2)$
is a $d_j$-sheeted covering.
The \emph{topological degree of $\Gamma$} is defined as
\[
  d_{top}(\Gamma)\,:=\,\sum\nolimits_{1\leq j\leq N} m_j\,d_j.
\]
In other words, $d_{top}(\Gamma)$ is the generic number of preimages of a point
counted according to multiplicity. If $X_1=X_2=X$ and $\Gamma = {\sf graph}(f)$, where
$f: X\lrarw X$ is a surjective holomorphic map, then the latter definition applied to ${\sf graph}(f)$
coincides with the classical definition of the topological degree of $f$.
\smallskip

One half of our proof of the following theorem
is strongly influenced by the derivation by Misiurewicz and Przytycki
\cite{misiurewiczPrzytycki:tedsm77} of a lower bound for topological
entropy of a single map in the $\smoo^1$ setting. Our notation below
follows the treatment of the above result by Katok--Hasselblatt in
\cite[Chapter~8]{katokHasselblatt:imtds95}. 

\begin{theorem}\label{T:general_bounds}
Let $(X, \omega)$ be a compact K{\"a}hler manifold of dimension $n$ and let  
$S$ be a finitely generated semigroup consisting of surjective holomorphic self-maps of $X$.
Let $\gen = \{f_1,\dots, f_N\}$ be a set of generators of $S$.
Then
\begin{align}
  \log\bigg(\sum\nolimits_{j=1}^N d_{top}(f_j)\bigg)\,&\leq\,h_{top}(S, \gen) \notag \\
  &\leq\,\max\bigg[\log(N),\,\log\bigg(\sum\nolimits_{j=1}^N d_{top}(f_j)\bigg),\,\max_{1\leq p\leq n-1}
  \log d_p(\Gamma_{\gen})\,\bigg]. \label{E:key_e_bounds}
\end{align}
\end{theorem}
\begin{proof}
Let $\omega_X$
be the normalisation of the form $\omega$ as in Section~\ref{S:prelim_complex}. 
For any holomorphic correspondence $\Gamma$ on $X$, we have (see
\cite[Section~3.1]{dinhSibony:dvtma06}, for instance):
\begin{equation}\label{E:d_top}
  \lambda_n(\Gamma)\,=\,\sum\nolimits_{1\leq j\leq N} m_j\!\!\int_{{\rm reg}(\Gamma_j)}
  \!\!\left(\pi_2|_{\Gamma_j}\right)^*\omega_X,
\end{equation}
assuming the presentation \eqref{E:stdForm_mult} for $\Gamma$. Since, for
each $j = 1,\dots, N$, $(\Gamma_j, \pi_2, X)$ is a holomorphic branched covering,
if follows from \eqref{E:d_top} and a change-of-variable argument that
$\lambda_n(\Gamma)$ equals the topological degree of $\Gamma$: call if $d_{top}(\Gamma)$.
Since the toplogical degree is multiplicative with respect to composition, it follows from
\eqref{E:d_deg} that
\[
  d_n(\Gamma)\,=\,\lim_{\nu\to \infty}d_{top}(\Gamma^{\circ \nu})^{1/\nu}\,=\,d_{top}(\Gamma)
\]
for any holomorphic correspondence $\Gamma$ on $X$. Applying this to the correspondence
$\Gamma_{\gen}$ we get
\begin{equation}\label{E:top_deg}
  d_n(\Gamma_{\gen})\,=\,\sum\nolimits_{1\leq j\leq N} d_{top}(f_j).
\end{equation}
A completely analogous discussion (whose details we leave to the reader) gives us
$d_0(\Gamma_{\gen})=N$. Recalling the definition of $h_{top}(S, \gen)$, 
the upper bound in \eqref{E:key_e_bounds} follows from the last identity,
\eqref{E:top_deg}, and Result~\ref{R:ent_u-bd}.
\smallskip

For any holomorphic map $f : X\lrarw X$, let $\jac(f)$ denote the real Jacobian
of $f$ determined by the volume form $\omega^n$. Since $f$ is holomorphic,
$\jac(f)\geq 0$. Fix a metric $d$ that metrises the topology of $X$. Fix a number $L$ such that
\[
  L\,>\,1 \quad\text{and}
  \quad \sup_{x\in X}\jac(f_j)(x)\,\leq\,L, \; \; j=1,\dots, N.
\]
Let us pick a number $\beta\in (0, 1)$ and set
$\delta(\beta) := L^{-\beta/(1-\beta)}$. Define the sets
\[
  \mathcal{B}(\beta, j)\,:=\,\{x\in X : \jac(f_j)(x)\geq \delta(\beta)\}, \quad j=1,\dots, N,
\]
and consider the open cover consisting of balls,
\[
  \mathscr{C}(\beta, j)\,:=\,\big\{B_d(x; r_x) : x\in \mathcal{B}(\beta, j) \text{ and }
  \left. f_j\right|_{B_d(x; r_x)} \text{ is invertible}\big\},
\]
of $\mathcal{B}(\beta, j)$, $j=1,\dots, N$. Let $\eps(\beta, j)\in (0, 1)$ be a Lebesgue
number of $\mathscr{C}(\beta, j)$ (each $\mathcal{B}(\beta, j)$ is compact) and write
$\eps(\beta) := \min_{1\leq j\leq N}\eps(\beta, j)$. 
\smallskip

Fix a $\nu\in \Z_+$. We simplify the symbol $\pth{\Gamma_{\gen}}_{\nu}$ to $\pth{\gen}_{\nu}$. 
For each $\wod \in \{1,\dots, N\}^{\nu}$, define
\[
  \A\,:=\,\big\{(x_0, x_1,\dots, x_{\nu}; \wod)\in \pth{\gen}_{\nu} \mid
  \sharp\{1\leq k\leq \nu : x_{k-1}\in \mathcal{B}(\beta, \alpha_k)\}\leq \beta\nu\big\}.
\]
For any $l$: $1\leq l\leq \nu$, let us abbreviate
\[
  f_{\alpha_l}\circ\dots\circ f_{\alpha_1}\,=:\,f_{(\alpha_1,\dots, \alpha_l)}.
\]
Whenever $\nu\geq 2$, the chain rule gives
\[
  \jac\big(f_{(\alpha_1,\dots, \alpha_\nu)}\big)(x)\,=\,\left[\,\prod_{2\leq k\leq \nu}
  \jac\big(f_{\alpha_k}\big)\big(f_{(\alpha_1,\dots, \alpha_{k-1})}(x)\big)\right]\jac\big(f_{\alpha_1}\big)(x).
\]
Therefore, by the definitions of $\A$ and $L$ (for any $\nu\in \Z_+$):
\begin{align}
 0\,\leq\,\jac\big(f_{(\alpha_1,\dots, \alpha_\nu)}\big)(x_0)\,&<\,\delta(\beta)^{\nu-\gti{\beta\nu}}
 		L^{\gti{\beta\nu}} \notag \\
 		&\leq \delta(\beta)^{\nu(1-\beta)}L^{\beta\nu}\,=\,1
 		 \quad \forall (x_0, x_1,\dots, x_{\nu}; \wod)\in \A, \label{E:size}
\end{align}
where here (and elsewhere in this proof) $\gti{s}$ denotes the greatest integer less than or equal to
$s$. If $\mu_X$ denotes the Borel measure, constructed in the standard manner, with the property that
$\mu_X(\OM) := \int_{\OM}\omega^n$ for every coordinate patch $\OM\subseteq X$, then:
\begin{enumerate}[label=$(\alph*)$, leftmargin=24pt]
  \item \eqref{E:size} implies that $\mu_X\big(f_{(\alpha_1,\dots, \alpha_\nu)}(\A)\big)
  < \mu_X(X)$ for each $\wod = (\alpha_1,\dots, \alpha_\nu)\in \{1,\dots, N\}^{\nu}$.
  
  \item\label{I:Sard} Thus, if we \textbf{fix} an $\wod$, then, by Sard's Theorem, there exists a point
  in $X\setminus f_{\wod}(\A)$ that is a regular value of $f_{\wod}$.
\end{enumerate} 
Let us call this regular value $x_{\nu}$.
\smallskip

For the $\wod = (\alpha_1,\dots, \alpha_{\nu})$ fixed in \ref{I:Sard} above, for any
$\alpha_k$, $1\leq k\leq \nu$, and any regular value $y$ of $f_{\alpha_{k}}$, we present a
construction associated with the pair $(y,k)$. Define
\[
  S(y, k)\,:=\,\begin{cases}
  			f_{\alpha_k}^{-1}\{y\}, &\text{if $f_{\alpha_k}^{-1}\{y\}\subset \mathcal{B}(\beta, \alpha_k)$}, \\
  			{} & {} \\
  			\{x^{(y)}\}, &\text{if $f_{\alpha_k}^{-1}\{y\}\not\subset \mathcal{B}(\beta, \alpha_k)$},
  			\end{cases}
\]
where, $x^{(y)}$ denotes some point in $f_{\alpha_k}^{-1}\{y\}\setminus \mathcal{B}(\beta, \alpha_k)$
that we pick and \textbf{fix}.
We now consider the point $x_{\nu}$ introduced at the end of the previous
paragraph. We will use it to construct a certain $(\eps(\beta), \wod)$-separated family in $\pth{\gen}_{\nu}$
using the following iterative construction. This construction is possible because, as every $f\in S$ is surjective,
by the definition of a regular value we get:
\[
  \text{each element of $(f_{\alpha_{\nu}}\circ\dots\circ f_{\alpha_{k+1}})^{-1}\{x_\nu\}$
  is a regular value of $f_{\alpha_k}$ for $1\leq k\leq \nu-1$.}
\]
Define (the maps appearing below were defined in Section~\ref{S:notn_top_ent}):
\begin{align*}
  {}^{1}O_{\wod}\,&:=\,\{(x, x_{\nu}; \alpha_{\nu}): x\in S(x_{\nu}, \nu)\}, \\
  {}^{k+1}O_{\wod}\,&:=\,\bigcup\nolimits_{\xi\in\,{}^kO_{\nu}}\big\{(x, \iti(\xi); 
  						\alpha_{\nu-k},\dots, \alpha_{\nu}):
  						x\in S(\sta(\xi), \nu-k)\big\}, \; \; 1\leq k\leq \nu-1.
\end{align*}
Here, we commit a minor abuse of notation in that if, for $1\leq k\leq \nu-1$,
${}^kO_{\nu}\ni\xi = (x_{\nu-k},\dots, x_{\nu}; \alpha_{\nu-k+1},\dots, \alpha_{\nu})$, then we
interpret $(x, \iti(\xi); \alpha_{\nu-k},\dots, \alpha_{\nu})$ to mean
\[
  (x, x_{\nu-k},\dots, x_{\nu}; \alpha_{\nu-k},\dots, \alpha_{\nu})
  \; \; \text{and \textbf{not}}\; \; (x, (x_{\nu-k},\dots, x_{\nu}); \alpha_{\nu-k},\dots, \alpha_{\nu}).
\]
With this explanation, note that each ${}^kO_{\nu}$ is a collection of $k$-orbits that end at the point
$x_{\nu}$. The iterative construction lengthens each $k$-orbit $\xi\in {}^kO_{\nu}$ to one or more $(k+1)$-orbits by
designating new initial points for the latter orbits.
\smallskip

Let us write $O_{\wod} := {}^{\nu}O_{\wod}$. We now show that $O_{\wod}$ is an
$(\eps(\beta), \wod)$-separated family of $\nu$-orbits. To do so, consider two
distinct $\nu$-orbits
\[
  {\sf x}:=(x_0,\dots, x_{\nu-1}, x_{\nu}; \wod) \quad\text{and}
  \quad {\sf y}:=(y_0,\dots, y_{\nu-1}, x_{\nu}; \wod)
\]
in $O_{\wod}$ (note that, by construction, the terminal points of these $\nu$-orbits are the same).
Write
\[
  \tau\,:=\,\max\{1\leq k\leq \nu : x_{k-1}\neq y_{k-1}\}.
\]
By our iterative construction, $x_{\tau-1}, y_{\tau-1}\in 
f_{\alpha_{\tau}}^{-1}\{x_{\tau}\}$. In terms of the notation introduced above, this also tells
us that $\sharp S(x_{\tau}, \tau)\geq 2$. This means that
\[
  x_{\tau-1}, y_{\tau-1}\in \mathcal{B}(\beta, \alpha_{\tau}) \quad\text{and}
  \quad\text{$f_{\alpha_{\tau}}$ is injective on small balls around  $x_{\tau-1}$, $y_{\tau-1}$}.
\]
Clearly, $x_{\tau-1}$ and $y_{\tau-1}$ cannot belong to one single ball belonging to
the open cover $\mathscr{C}(\beta, \alpha_{\tau})$. Thus, by the definition of Lebesgue number,
$d(x_{\tau-1}, y_{\tau-1}) > \eps(\beta)$. Since ${\sf x}\neq {\sf y} \in 
O_{\wod}$ were arbitrarily chosen, we conclude that $O_{\wod}$ is an
$(\eps(\beta), \wod)$-separated family.
\smallskip

Write $d_j := d_{top}(f_j)$, $j=1,\dots, N$. We may assume without loss of generality
that $d_1\leq\dots \leq d_N$. Let $\nu_j :=\ $the number of times $j$ appears in $\wod$.
Let us now set
\begin{equation}\label{E:m_and_p}
  m\,:=\,\gti{\beta\nu}+1, \qquad
  J\,:=\,\max\{1\leq j\leq N : \nu_1+\dots +\nu_j < m\}.
\end{equation}
By construction, for each $\nu$-orbit ${\sf x}\in O_{\wod}$,
$\sta(x)\in f_{\wod}^{-1}\{x_{\nu}\}$. Hence, by our above choice of $x_{\nu}$,
\begin{equation}\label{E:not_in_A}
  \{\sta(\bsy{x}) : \bsy{x}\in O_{\wod}\}\cap \A\,=\,\varnothing.
\end{equation}
Some more notation: write
\[
  \Sigma(\bsy{x})\,:=\,\{1\leq k\leq \nu: x_{k-1}\in \mathcal{B}(\beta, \alpha_k)\},
  \qquad \sigma(\bsy{x})\,:=\,\sharp\Sigma(\bsy{x})
\] 
for each $\bsy{x} = (x_0, x_1,\dots, x_{\nu}; \wod)\in O_{\wod}$. Additionally,
let $k_1 < k_2 <\dots < k_{\sigma(\bsy{x})}$ denote the ordering of the elements
of $\Sigma(\bsy{x})$. By \eqref{E:not_in_A}, $\sigma(\bsy{x})\geq m$ for each  
$\bsy{x}\in O_{\wod}$. With these facts, we can estimate $\sharp O_{\wod}$. To
do so, pick and \textbf{fix} an $\bsy{x} = (x_0, x_1,\dots, x_{\nu}; \wod)
\in O_{\wod}$. By construction:
\[
  \sharp S(x_k, k)\,=\,\begin{cases}
  					1, &\text{if $k\notin \Sigma(\bsy{x})$}, \\
  					d_{\alpha_k}, &\text{if $k\in \Sigma(\bsy{x})$}.
  					\end{cases}
\]
This means that in $O_{\wod}$\,:
\begin{itemize}[leftmargin=24pt]
  \item[$(*)$] we can find
  $d_{{\raisebox{1pt}{$\scriptstyle{\alpha}$}}_{k_l}}$ 
  distinct $\nu$-orbits that traverse the
  points $x_{k_l}, x_{k_l + 1},\dots, x_{\nu}\in X$ corresponding, respectively, to
  iterations of orders $k_l, k_l + 1,\dots, \nu$ of $\Gamma_{\gen}$,
  $l = 1,\dots \sigma(\bsy{x})$.
\end{itemize}
This implies that $O_{\wod}$ would have the smallest possible number of orbits of
the kind described by $(*)$ if $\nu_1+\dots +\nu_J =: d(\wod)$ of the elements of
$\Sigma(\bsy{x})$ were to correspond to
$\nu_j$ distinct terms in the tuple $(x_0,\dots, x_{\nu-1})$ being in
$\mathcal{B}(\beta, j)$, $j = 1,\dots, J$. From this discussion and $(*)$, we get the
(perhaps very conservative) lower bound:
\begin{equation}\label{E:O_est}
  n(\eps(\beta), \wod)\,\geq\,\sharp O_{\wod}\,\geq\,d_1^{\nu_1}
  \dots\,d_J^{\nu_J}\,d_{J+1}^{m-d(\wod)}.
\end{equation}
Here, $n(\eps(\beta), \wod)$ is as in Lemma~\ref{L:sum-up}, and the first
inequality in \eqref{E:O_est} is owing to the fact that $O_{\wod}$ is
$(\eps(\beta), \wod)$-separated. 
\smallskip

Now, given any $(\nu_1,\dots, \nu_N)\in \N^{N}$ satisfying
$\nu_1 + \dots + \nu_N = \gti{\beta\nu}+1$, we can find an
$\wod\in \{1,\dots, N\}^{\nu}$ so that the $\nu_j$'s are related to this
$\wod$ precisely as in the last paragraph. Thus:
\begin{align*}
  \sum_{\wod\in \{1,\dots, N\}^{\nu}}\!\!n(\eps, \wod)\,
  &\geq\!\sum_{\genfrac{}{}{0pt}{3}{\nu_1,\dots, \nu_N\in \N}{\nu_1+\dots +\nu_N = m}}
  								d_1^{\nu_1}\dots d_N^{\nu_N} \\
  &=\,(d_1+\dots +d_N)^m\,\geq\,(d_1+\dots +d_N)^{\beta\nu}.
\end{align*}
Applying Lemma~\ref{L:sum-up}, this gives
\[
  h_{top}(S, \gen)\,:=\,h_{top}(\Gamma_{\gen})\,\geq\,\beta\log(d_1+\dots +d_N).
\]
However, as this holds for any $\beta\in (0,1)$, letting $\beta\to 1^{-}$, we
get
\[
  h(S, \gen)\,\geq\,\log(d_1+\dots d_N).
\]
This establishes the lower bound in \eqref{E:key_e_bounds}, and hence the result.
\end{proof}

We remark here that it is, in general, not possible to get a cleaner upper bound for
$h_{top}(S, \gen)$ than \eqref{E:key_e_bounds}. For instance, there isn't, in general, a way to
determine which of the numbers $\{d_1(f),\dots, d_n(f)\}$ is the largest
even for $f : X\lrarw X$ surjective and holomorphic (let alone for
a general correspondence $\Gamma$). We shall not discuss here what \textbf{is}
known in general about the function $\{0, 1,\dots, n\}\ni p\mapsto d_p(f)$. However,
for $X$ as above, $\lambda_p(f)$, for $f : X\lrarw X$ holomorphic and $p=0, 1,\dots, n$,
can be determined cohomologically. This can lead to cleaner expressions whenever
$H^{p, p}(X; \R)$ are one-dimensional for each $p=1,\dots, n$. This, essentially,
is what underlies

\begin{proof}[The proof of Theorem~\ref{T:main_D-S}]
Russakovskii--Shiffman have shown \cite[Section~4]{russaSchiff:vdsrcd97} that for any
non-constant holomorphic map $f: \pro^n\lrarw \pro^n$
\begin{equation}\label{E:russ_schiff}
  \lambda_p(f)\,=\,d_p(f) \quad\text{and}
  \quad d_p(f)\,=\,d_1(f)^p \quad\text{for $p = 1,\dots, n$}.
\end{equation}
As argued in the proof of Theorem~\ref{T:general_bounds}, $d_n(f)=d_{top}(f)$. Thus,
from the above facts, we get
\begin{equation}\label{E:d_top_proj}
  \lambda_p(f)\,=\,d_{top}(f)^{p/n},
  \quad\text{for $p = 1,\dots, n$}.
\end{equation}
Fix a set of generators $\{f_1,\dots, f_N\}$ of $S$. Clearly, by definition, 
$\lambda_p(\Gamma_{\gen}^{\circ \nu})$ is the sum of the $\ordi{p}$ intermediate
degrees of the maps, counted according to multiplicity, whose graphs constitute
$\Gamma_{\gen}^{\circ\nu}$. From \eqref{E:d_top_proj}, we see that for $\pro^n$,
$\lambda_p$, $p=1,\dots, n$,
is multiplicative with respect to composition of non-constant holomorphic self-maps.
Thus, by \eqref{E:russ_schiff}, we get
\[
  \lambda_p(\Gamma_{\gen}^{\circ \nu})\,=\,\big(d_1(f_1)^p +\dots + d_1(f_N)^p\big)^\nu
  \quad\text{for $p = 1,\dots, n$}.
\]
Hence, $d_p(\Gamma_{\gen}) = \big(d_1(f_1)^p +\dots + d_1(f_N)^p\big)$,
$p = 1,\dots,n$. As for $p=0$: $d_0(\Gamma_{\gen}) = N$. Given these facts, the
conclusion of Theorem~\ref{T:main_D-S} follows from Theorem~\ref{T:general_bounds}.    
\end{proof}

Corollary~\ref{C:one-dim_D-S} now follows immediately. 
\smallskip

\section{Concerning Friedland's entropy}\label{S:main_F}

This section is dedicated to the result on $h_F$ mentioned in 
Section~\ref{S:intro}\,---\,i.e., Theorem~\ref{T:F_entr}. First, however,
we need some notation and a lemma.
Let $S$ be a finitely generated rational semigroup on $\pro^1$. If
we fix a finite set of generators $\gen$, then the space $\Gamma^{\infty}$
(introduced in Definition~\ref{D:entropy_F}) corresponding to $\Gamma_{\gen}$
will be denoted by $\pat$. Also,
we abbreviate $\pathsp{\Gamma_{\gen}}$ to $\pathsp{\gen}$. Given any
holomorphic correspondence $\Gamma$ from $X_1$ to $X_2$, where
$X_i$, $i=1, 2$, are as in Definition~\ref{D:holCorr}, we define
\[
  F_{\Gamma}(x)\,:=\,\pi_2\!\left(\big(\left.\pi_1\right|_{|\Gamma|}\big)^{-1}\{x\}\right)
  \quad \forall x\in X_1,
\]
and write $F_{\Gamma}^{\nu}(x):=F_{\Gamma^{\circ\nu}}(x)$. Coming back to the correspondence
$\Gamma_{\gen}$: we abbreviate $F_{\Gamma_{\gen}}^{\nu}(x)$ to $F_{\gen}^{\nu}(x)$.
   
\begin{lemma}\label{L:limsup}
Let $S$ be a finitely generated rational semigroup on $\pro^1$.
Let $\gen = \{f_1,\dots, f_N\}$ be
a set of generators of $S$. Write
\[
  \E\,:=\,\bigcup\nolimits_{j=1}^N\bigcup\nolimits_{i\neq j}\{x\in \pro^1 : f_i(x)=f_j(x)\}.
\]
Consider a point $\mathscr{O} = (x_0, x_1, x_2,\dots)\in \pat$. If the
pre-image of $\mathscr{O}$ under the map $\iti: \pathsp{\gen}\lrarw \pat$ is
infinite, then there exist an $n^{\bullet}\in \N$ and a point
\[
  x^{\bullet}\in \E\bigcap \Big(\limsup_{\nu\to \infty}F_{\gen}^{\nu}(x^{\bullet})\Big)
\]
such that $x^{\bullet}=x_{n^{\bullet}}$. 
\end{lemma}
\begin{proof}
First consider any $\mathscr{O} = (x_0, x_1, x_2,\dots)\in \pat$. For each $\nu\in \Z_+$,
there are only finitely many $j$ such that $f_j(x_{\nu-1}) = x_{\nu}$. Thus, it is easy to see that
$\iti^{-1}\{\mathscr{O}\}$ is infinite if and only if
\begin{itemize}
  \item[$(**)$] there exists a sequence of positive integers $\nu_1 < \nu_2 < \nu_3~\!\!<~\!\!\cdots$
  such that for each $k\in \Z_+$, $f_j(x_{(\nu_k-1)}) = x_{\nu_k}$ for \textbf{more than}
  one $j\in \{1,\dots, N\}$.
\end{itemize}
Now, assume that $\iti^{-1}\{\mathscr{O}\}$ is infinite. Then, by $(**)$ there exists a sequence of
positive integers
$\nu_1 < \nu_2 < \nu_3~\!\!<~\!\!\cdots$ such that
\begin{equation}\label{E:in_S}
  x_{(\nu_k-1)}\in \E \quad \forall k\in \Z_+.
\end{equation}
Since $f_1,\dots, f_N$ are distinct and $\pro^1$ is one-dimensional, $\E$ is finite.
Thus, by \eqref{E:in_S}, we conclude that there exists an increasing subsequence
$\{\nu_{k_{\ell}}\}\subset \{\nu_k\}$ and a point $x^{\bullet}\in \E$ such that
\[
  x_{(\nu_{k_{\ell}}-1)}\,=\,x^{\bullet}\in \E
  \quad \forall \ell\in \Z_+.
\]
If we write $n_{\ell}:=\nu_{k_\ell}-\nu_{k_1}$, then the above equation implies
that $x^{\bullet}\in F_{\gen}^{n_{\ell}}(x^{\bullet})$ for every $\ell\in \Z_+\setminus\{1\}$.
Therefore, we conclude that
\[
  x^{\bullet}\in\bigcap\nolimits_{k\in \N}\bigcup\nolimits_{\nu\geq k}
  F_{\gen}^{\nu}(x^{\bullet})\,=:\,\limsup_{\nu\to \infty}F_{\gen}^{\nu}(x^{\bullet}).
\]
Taking $n^{\bullet}=(\nu_{k_1}-1)$, the desired conclusion is obtained.
\end{proof}

Before we present Theorem~\ref{T:F_entr}, we elaborate upon the remark made towards the end
of Section~\ref{S:intro}.
For the purposes of this discussion, let $X$ be any compact
metric space and let $S$ be the semigroup generated by the maps $f_j: X\lrarw X$, $j=1,\dots,N$,
that are \emph{continuous}, and take $\Gamma = \cup_{1\leq j\leq N}\,{\sf graph}(f_j)$
in Definition~\ref{D:entropy_F}. For each $A\varsubsetneq \{1,\dots, N\}$, $A\neq \varnothing$,
consider
\[
  Y^{(A)}\,:=\,\{(x_0, x_1, x_2,\dots)\in \Gamma^\infty: x_{n+1}=f_j(x_n)
  \ \text{for some $j\in A$}, \ n=0, 1, 2,\dots\}.
\]
$Y^{(A)}$ is a closed subspace of $\Gamma^\infty$ that is invariant under $\sigma$, where
$\sigma$ is as in Definition~\ref{D:entropy_F}. Recalling the definition of 
$h_F(S, \gen)$, the basic properties of Bowen's entropy, and as $Y^{(A)}$ is
$\sigma$-invariant, we get (as before, $\gen := \{f_1,\dots, f_N\}$)
\[
  h_F(S, \gen)\,\geq\,h\big(\left.\sigma\right|_{Y^{(A)}}
  \big)\,=\,h_F(\langle f_j: j\in A\rangle,\,\{f_j: j\in A\}).
\]
When $A=\{j\}$, write $Y^{(A)} =: Y^{(j)}$. Observe:
$f_j$ is conjugate to $\left.\sigma\right|_{Y^{(j)}}$ via the map
$x\longmapsto (x, f_j(x), f_j\circ f_j(x),\dots)\in Y^{(j)}$. From this and our
preceding argument,
we get
\[
  h_F(S, \gen)\,\geq\,h\big(\left.\sigma\right|_{Y^{(j)}}\big)\,=\,h(f_j)
  \quad \forall j=1,\dots,N.
\]
Hence, $h_F(S, \gen)\geq \max_{1\leq j\leq N}h(f_j)$. In particular,
this estimate holds true for all the semigroups discussed in
Sections~\ref{S:intro},~\ref{S:main_D-S} and the present section.
\smallskip

One might intuit from the previous lemma that the latter lower bounds could be improved
(as least when $S$ is a finitely generated rational semigroup on $\pro^1$). That
intuition motivates the principal result of this section. We follow below the notation established for
Proposition~\ref{P:ent_ed}\,---\,for instance, $\sh$ is the shift map introduced by that
proposition.

\begin{theorem}\label{T:F_entr}
Let $S$ be a finitely generated rational semigroup on $\pro^1$. Let
$\gen = \{f_1,\dots, f_N\}$ be a set of generators of $S$. Define
\[
  \excp\,:=\,\Big\{(x_0, x_1, x_2,\dots)\in \pat : x_0\in \E \;\,\text {and} 
  \;\,x_0\in \limsup_{\nu\to \infty}F_{\gen}^{\nu}(x_0)\Big\}.
\]
Then, Friedland's entropy satisfies
\begin{align}
  \log\left(\,\sum\nolimits_{1\leq j\leq N} \deg(f_j)\right) - 
  \sup_{\mathscr{O}\in \excp}h\big(\sh, \iti^{-1}\{\mathscr{O}\}\big)\,&\leq\,h_{F}(S, \gen) \notag \\
  &\leq\,\log
  \left(\,\sum\nolimits_{1\leq j\leq N} \deg(f_j)\right). \label{E:F_entr_bounds}
\end{align}
\end{theorem}

Before seeing a proof of the above theorem, it might be helpful to discuss
when the estimates \eqref{E:F_entr_bounds} are informative. To this end, we refer to the paragraph
just after the proof of Lemma~\ref{L:limsup}. We see there a lower bound that applies to
any finitely generated rational semigroup on $\pro^1$. To the best of our knowledge, this (apart from
the exact computations for semigroups of M{\"o}bius transformations in \cite{friedland:egsg96}) is the
only lower bound known for $h_{F}(S, \gen)$. With this in mind: an example of a class of $(S, \gen)$ for which
\eqref{E:F_entr_bounds} would be informative is a class comprising finitely generated semigroups
of M{\"o}bius transformations with $\sharp\gen\geq 3$ such that the left-hand side of \eqref{E:F_entr_bounds}
is positive (however small). One way to achieve this, for $(S, \gen)$ as described, is for the maps in each proper
subset of $\gen$ of cardinality $(\sharp\gen -1)$ to have common fixed points (along with a technical condition)
but for the maps in $\gen$ to have no fixed points in common. This endows the elements in $\excp$ with
a very specific structure. However, since
\begin{itemize}[leftmargin=24pt]
  \item a detailed discussion of such an example would be rather protracted, and
  \item we do not see an overarching geometric description for (the class of) such semigroups,
\end{itemize}
we shall not dwell any further on this. That being said, there is an alternative approach that involves
establishing that $(\pathsp{\gen}, \mu)$ is isomorphic to $\big(\pat, (\iti)_*\mu\big)$ for an appropriate
Markov measure $\mu$, which accounts more coherently for the $(S,\gen)$'s alluded to above. This is currently work
in progress by the first-named author; details will appear elsewhere.
\smallskip

We present an example where the lower bound in \eqref{E:F_entr_bounds} is informative for a different reason, and
in which we get to see \textbf{some} of the reasoning hinted at above. Consider the semigroup $S$ whose set of generators
$\gen$ comprises the two functions $f_1$ and $f_2$ such that
\[
  \left. f_j\right|_{\C}:\,z\longmapsto z+a_j \quad  j=1,2,
\]
where $a_1\neq a_2\in \C$. So, $\E=\{\infty\}$. Moreover, as $\infty$ is the common
fixed point of $f_1$ and $f_2$, $\excp$ is a singleton consisting of the constant sequence
$\mathscr{O}_{\infty}\:=(\infty, \infty,\dots)$. Thus
\[
  \iti^{-1}\{\mathscr{O}_{\infty}\}\,=\,\{(\infty, \infty, \infty\dots; \alpha_1, \alpha_2,\dots): 
  \alpha_j\in \{1, 2\} \ \text{for each $j\in \Z_+$}\}.
\]
Observe that $\sh(\iti^{-1}\{\mathscr{O}_{\infty}\})\subseteq \iti^{-1}\{\mathscr{O}_{\infty}\}$.
Thus, the iterated application of $\sh$ on $\iti^{-1}\{\mathscr{O}_{\infty}\}$ behaves like the
dynamics of the left-shift on $\{1,2\}^{\N}$, whence we get
\[
  h(\sh, \iti^{-1}\{\mathscr{O}_{\infty}\})\,=\,\log(2).
\]
Thus, \eqref{E:F_entr_bounds} gives $0$ as the lower bound for $h_F(S, \gen)$.
However, there is \textbf{no better} lower bound for $h_F(S, \gen)$ in the vast majority of the examples
in this class! Specifically, if $\{a_1, a_2\}$ is linearly independent over $\R$,
or if $a_2 = ca_2$ for some $c\geq 0$, then $h_F(S, \gen)=0$\,---\,see \cite[Lemma~5.3]{friedland:egsg96}.
With those comments on \eqref{E:F_entr_bounds}, we now present

\begin{proof}[The proof of Theorem~\ref{T:F_entr}]
Consider the compact metric spaces $\pathsp{\gen}$ and $\pat$. Denote by
$\Sig: \pat\lrarw \pat$ the shift
\[
  \pat\ni (x_0, x_1, x_2,\dots)\mapsto (x_1, x_2, x_3,\dots).
\]
Recall that $\sh: \pthsp{\gen}\lrarw\pthsp{\gen}$ is as described in
Proposition~\ref{P:ent_ed}. We have:
\begin{itemize}[leftmargin=24pt]
  \item $\iti$ is a continuous surjective map; and
  \item $\iti\circ\sh = \Sig\circ\iti$.
\end{itemize}
In other words, $\Sig$ is a factor of $\sh$. Thus, Proposition~\ref{P:ent_ed}
and Corollary~\ref{C:one-dim_D-S} together imply:
\begin{equation}\label{E:up_bound}
  h_F(S, \gen)\,:=\,h(\Sig)\,\leq\,h(\sh)\,=\,h_{top}(S, \gen)\,=\,\log
  \left(\,\sum\nolimits_{1\leq j\leq N} \deg(f_j)\right).
\end{equation}

We now derive the lower bound for $h_F$.
For the moment, \textbf{fix} an $\mathscr{O}\in \pat$. We have two cases.%
\medskip

\noindent{\textbf{Case~1.} \emph{$\iti^{-1}\{\mathscr{O}\}$ is a finite set.}}
\vspace{0.6mm}

\noindent{For any $\nu\in \Z_+$ and $\eps>0$, $\mathscr{O}$ $(\eps, \nu)$-spans 
itself. Thus, the finiteness of $\iti^{-1}\{\mathscr{O}\}$ implies that
$h(\sh,\,\iti^{-1}\{\mathscr{O}\}) = 0$.}
\medskip

\noindent{\textbf{Case~2.} \emph{$\iti^{-1}\{\mathscr{O}\}$ is an infinite set.}} 
\vspace{0.6mm}

\noindent{Write $\mathscr{O} = (x_0, x_1, x_2,\dots)$. In this case, Lemma~\ref{L:limsup}
enables us to define
\[
  k(\mathscr{O})\,:=\,\min\Big\{n\in \N: x_n\in
  \E\cap \big(\limsup_{\nu\to \infty}F_{\gen}^{\nu}(x_n)\big)\Big\}.
\]
If $k(\mathscr{O})\geq 1$, then by definition, there is a fixed tuple 
$(\alpha_1,\dots, \alpha_{k(\mathscr{O})})$ such that
every element of $\iti^{-1}\{\mathscr{O}\}$ has the form
  $(x_0,\dots, x_{k(\mathscr{O})},\dots; \alpha_1,\dots, \alpha_{k(\mathscr{O})},\dots)$.
Thus, following the
notation of the discussion that precedes Result~\ref{R:bowen_quots},
for each $\eps>0$ we have:
\[
  r_{\nu}\big(\eps, \iti^{-1}\{\mathscr{O}\}\big)
  = r_{\nu}\big(\eps, \iti^{-1}\{(x_{k(\mathscr{O})},\,
  x_{k(\mathscr{O})+1},\,x_{k(\mathscr{O})+2}\,,\dots)\}\big) \quad \forall \nu \ \text{suffiiciently large.}
\]
We therefore conclude (irrespective of whether $k(\mathscr{O})=0$
or $k(\mathscr{O})\geq 1$) that
\begin{align}
  h\big(\sh, \iti^{-1}\{\mathscr{O}\}\big)\,&=\,h\big(\sh,\iti^{-1}\{(x_{k(\mathscr{O})},\,
  x_{k(\mathscr{O})+1},\,x_{k(\mathscr{O})+2}\,,\dots)\}\big), \; \;\text{and} \notag \\
  x_{k(\mathscr{O})}\,&\in\,\E
  \cap \big(\limsup_{\nu\to \infty}F_{\gen}^{\nu}(x_{k(\mathscr{O})})\big). \label{E:orbi}
\end{align}}

From the discussion of each of the above cases, and by \eqref{E:orbi}, we get
\[
  \sup_{\mathscr{O}\in \pat}h\big(\sh, \iti^{-1}
  \{\mathscr{O}\}\big)\,=\,\sup_{\mathscr{O}\in \excp}h\big(\sh, \iti^{-1}\{\mathscr{O}\}\big).
\]
From this, Result~\ref{R:bowen_quots} and Proposition~\ref{P:ent_ed}, we have
\begin{align}
  h_{top}(S, \gen)\,=\,h(\sh)\,&\leq\,h(\Sig) +
  \sup\nolimits_{\mathscr{O}\in \excp}h\big(\sh, \iti^{-1}\{\mathscr{O}\}\big) \notag \\
  &=\,h_{F}(S, \gen) +
  \sup\nolimits_{\mathscr{O}\in \excp}h\big(\sh, \iti^{-1}\{\mathscr{O}\}\big).  \label{E:lo_bound}
\end{align}
From \eqref{E:up_bound} and \eqref{E:lo_bound}, and given the conclusion of
Corollary~\ref{C:one-dim_D-S}, the theorem follows.
\end{proof}

\begin{remark}
Theorem~\ref{T:F_entr} can be extended to higher dimensions. But, in that case, it is possible for the
set $\E$ introduced in Lemma~\ref{L:limsup}
to contain algebraic varieties of positive dimension. In view of the discussion preceding the proof of
Theorem~\ref{T:F_entr}, it seems likely that it would be difficult to understand the analogue of the
lower bound of $h_{F}(S, \gen)$ in higher dimensions. Therefore, we have focused on classical rational
semigroups in this section.
\end{remark}

\section*{Acknowledgments}
\noindent{Gautam Bharali is supported by a Swarnajayanti Fellowship (Grant no.\;DST/SJF/MSA-02/2013-14)
and a UGC CAS-II grant (Grant No. F.510/25/CAS-II/2018(SAP-I)).}

\end{document}